\documentclass{amsart}
\usepackage{amsmath}
\usepackage{amsthm}
\usepackage{amssymb}
\usepackage{amsfonts}
\usepackage{graphicx}
\usepackage{paralist}
\usepackage{hyperref}

\newtheorem{PROP}{Proposition}[section]
\newtheorem{THM}[PROP]{Theorem}
\newtheorem{LM}[PROP]{Lemma}
\newtheorem{COR}[PROP]{Corollary}
\theoremstyle{definition}
\newtheorem{DEF}[PROP]{Definition}
\newtheorem{EXA}[PROP]{Example}
\newtheorem{REM}[PROP]{Remark}

\newcommand{\Zobr}[3]{\ensuremath{#1\colon #2\to #3}}
\newcommand{\inv}[1]{#1^{-1}}
\newcommand{\Invobr}[2]{\inv{#1}(#2)}
\newcommand{\Obr}[2]{#1[#2]}
\newcommand{\mc}[1]{\mathcal{#1}}
\newcommand{\abs}[1]{|{#1}|}
\newcommand{\ol}[1]{\overline{#1}}
\newcommand{\Dots}[3]{\ensuremath{#1_{#2}\ldots#1_{#3}}}
\newcommand{\emps}{\emptyset}
\newcommand{\sm}{\setminus}
\newcommand{\RR}{\mathbb{R}}
\newcommand{\NN}{\mathbb{N}}
\newcommand{\Ra}{\Rightarrow}

\newcommand{\ve}{\varepsilon}
\newcommand{\vk}{\varkappa}

\newcommand{\omone}{\ensuremath{\omega_1}}

\newcommand{\uv}[1]{``#1''}

\newcommand{\enu}{\renewcommand{\theenumi}{\roman{enumi}}\renewcommand{\labelenumi}{\textrm{(\theenumi)}}}

\newcommand{\FF}{\mathcal{F}}
\newcommand{\FI}{\mathcal{F}(\I)}
\newcommand{\FJ}{\FF(\J)}
\newcommand{\FD}{\FF_D}
\newcommand{\APIJ}{\mathrm{AP}(\I,\J)}
\newcommand{\APIFin}{\mathrm{AP}(\I,\Fin)}
\newcommand{\APIdFin}{\mathrm{AP}(\Id,\Fin)}

\newcommand{\I}{\ensuremath{\mathcal{I}}}
\newcommand{\ID}{\ensuremath{\I_D}}
\newcommand{\J}{\mathcal{K}}
\newcommand{\Fin}{\mathrm{Fin}}
\newcommand{\Id}{\ensuremath{\mathcal{I}_d}}
\newcommand{\IhJ}{\IhJh{\I}{\J}}
\newcommand{\IhI}{\IhJh{\I}{\I}}
\newcommand{\IhJh}[2]{\ensuremath{#1^{#2}}}

\newcommand{\powerset}[1]{\mc P(#1)}

\begin{document}

\title{$\mathcal I^{\mathcal K}$-Cauchy functions}


\author{Pratulananda Das}
\address{Department of Mathematics, Jadavpur University, Kolkata-700032, West Bengal, India}
\email{\tt pratulananda@yahoo.co.in}

\author{Martin Sleziak}
\address{Department of Algebra, Geometry and Mathematical Education, Faculty of Mathematics, Physics and Informatics, Comenius University, Mlynsk\'a dolina, 842 48 Bratislava, Slovakia}
\email{\tt sleziak@fmph.uniba.sk}

\author{Vladim\'\i r Toma}
\address{Department of Mathematical Analysis and Numerical Mathematics, Faculty of Mathematics, Physics and Informatics, Comenius University, Mlynsk\'a dolina, 842 48 Bratislava, Slovakia}
\email{\tt toma@fmph.uniba.sk}

\date{\today}

\maketitle

\begin{abstract}
In this paper we introduce the notion of $\mathcal I^{\mathcal K}$-Cauchy function, where $\mathcal I$ and $\mathcal K$ are ideals on the same set.
The $\mathcal I^{\mathcal K}$-Cauchy functions are a generalization of $\mathcal I^*$-Cauchy sequences and $\mathcal I^*$-Cauchy nets. We show how this notion can be used to characterize complete uniform spaces and we study how $\mathcal I^{\mathcal K}$-Cauchy functions and $\mathcal I$-Cauchy functions are related. We also define and study $\mathcal I^{\mathcal K}$-divergence of functions.
\end{abstract}

\noindent Keywords: ideal convergence,  filter, Cauchy net, $\mathcal I$-Cauchy condition

\noindent Mathematical Reviews subject classification: Primary 54A20; Secondary 40A05.

\section{Introduction}

The main goal of this paper is to study the $\IhJ$-Cauchy condition, which is a common generalization of various types of $\I^*$-Cauchy condition. As the research of $\I^*$-Cauchy sequences and nets is tightly linked to two types of ideal convergence, let us start by briefly mentioning the history of them.

The notions of $\I$-convergence and $\I^*$-convergence, where $\I$ is an ideal on the set $\NN$, were introduced in \cite{KMS} and \cite{KSW}. These two notions give a rather natural generalization of the notion of statistical convergence. In fact, in the case that $\I=\I_d$ is the ideal consisting of sets having asymptotic density equal to zero, these two types of convergence are equivalent. In \cite{KSW} the authors characterized ideals on the set $\NN$ for which $\I$-convergence and $\I^*$-convergence of sequences in metric spaces are equivalent. Similar results for the first countable topological spaces were obtained in \cite{LAHDASTOP}.

These two types of convergence were studied in several other contexts, too. In \cite{DASKOSMAWI} the authors defined $\I$-convergence and $\I^*$-convergence of double sequences and they studied the question when these two types of convergence coincide.
The paper \cite{LAHDASNETS} was focused on the $\I$-convergence and $\I^*$-convergence of nets.
In \cite{GEZERKARAKUS} and \cite{BADEKO} the $\I$-convergence and $\I^*$-convergence of sequences of functions were examined.

Later the authors of the paper \cite{MACSLEIK} defined the $\IhJ$-convergence and they showed that this type of convergence is a common generalization for all types of $\I$- and $\I^*$-convergence we have mentioned so far. Also a characterization of the ideals $\I$ and $\J$ such that $\I$-convergence is equivalent to $\IhJ$-convergence was given. It is perhaps not very surprising that these results encompass, as special cases, the results obtained before for double sequences, nets and functional sequences.

When studying the convergence of sequences, several closely related notions occur quite naturally, such as cluster points, Cauchy sequences, limit superior, subsequences etc. Corresponding concepts appear when dealing with convergence of nets and filters. After defining some type of convergence, it is very natural to ask whether analogues of these concepts can be defined for this particular type of convergence and whether they have similar properties as for sequences, nets, filters. In this paper we will concentrate on the version of Cauchy condition, which corresponds to the $\IhJ$-convergence.

Again, special cases of the $\IhJ$-Cauchy condition have been studied by several authors in various contexts.
Fridy introduced the statistical Cauchy condition in \cite{FRIDY}. The statistically Cauchy sequences were also discussed in \cite{DIMAIOKOCINACSTAT}.
The $\I$-Cauchy sequences and $\I^*$-Cauchy sequences in metric spaces were studied in \cite{NABPEHGUR}.
In \cite{DASGHOSALCOMPLETE} the authors introduced the notion of $\I^*$-Cauchy nets. $\IhJ$-Cauchy nets were introduced in \cite{DASPALIKCAUCHYNETS} and some results similar to the results obtained in this paper were given there.
The $\I$-Cauchy condition for double sequences was given in \cite{DEMSICAUCH}, although $\I^*$-Cauchy condition was not studied there.

We will show that the notion of $\IhJ$-Cauchy function introduced in this paper is a common generalization of the above mentioned concepts. We will show generalizations of several results which were obtained for these special cases.

We start by defining the $\IhJ$-Cauchy function, the definition is followed by several relatively straightforward observations about $\IhJ$-Cauchy functions and $\IhJ$-convergent functions. Then we study the relation between this notion and completeness. It is relatively easy to show that
if $X$ is a complete space, then every $\IhJ$-Cauchy function $\Zobr fSX$ is also $\IhJ$-convergent. We also discuss when the converse implication is true.

The next topic is the study of the relationship between $\I$-Cauchy functions and $\IhJ$-Cauchy functions. We show that in metric spaces these two notions are equivalent if and only if the ideals fulfill the condition $\APIJ$, which was introduced in \cite{MACSLEIK}. We also include examples showing that this characterization is not true for arbitrary uniform spaces.

In the last section of this paper we define $\IhJ$-divergence as a generalization of $\I^*$-divergence, which was studied in \cite{DASGHOSAL}. We show, how we can obtain some results on $\IhJ$-divergence and $\I^*$-divergence from the results about $\IhJ$-convergence.

\section{Basic definitions and notation}

\begin{DEF}
An \emph{ideal} (of sets) is a nonempty system of sets $\I$ hereditary with respect to inclusion and additive i.e.
\begin{compactenum}
 \enu
	\item $A\subseteq B\in\I \Longrightarrow A\in\I$;
	\item $A, B \in \I \Longrightarrow  A\cup B \in \I$;
\end{compactenum}
\end{DEF}

If the elements of $\I$ are subsets of a given set $S$ we say that $\I$ is an \emph{ideal on the set} $S$.

If $\I\subsetneq \mathcal{P}(S)$ we say hat $\I$ is a \emph{proper ideal on S}.

If for every $x\in S$ the singleton $\{x\}\in\I$ we say that $\I$ is an \emph{admissible ideal on $S$}. The smallest admissible ideal on a set $S$ is the  system $\Fin(S)$ of all finite subsets of $S$. In the case $S=\NN$ we will write $\Fin$ instead of $\Fin(\NN)$ for short.

The dual notion to the ideal is the notion of the filter. I.e., a \emph{filter on $S$} is a non-empty system of subsets of $S$, which is closed under supersets and finite intersections.

A filter $\FF$ is called \emph{free} if $\bigcap\FF=\emptyset$.

A filter $\FF$ on a set $S$ is called \emph{proper} if $\FF\ne\powerset{S}$ (or, equivalently, $\emps\notin\FF$).

For every ideal $\I$ on $S$ we have the corresponding dual filter $\FI=\{S\setminus A: A\in \I\}$ on the same set $S$.

If $\I$ is an ideal on $S$ and $M\subseteq S$ then we denote by $\I|_M$ the trace of the ideal $\I$ on the subset $M$, i.e.
$$\I|_M=\{A\cap M; A\in\I\}.$$
The dual filter is $\FF(\I|_M)=\{F\cap M; F\in\FI\}$.

If $\Zobr fSX$ is a map and $\FF$ is a filter on the set $S$, then $\Obr f{\FF}=\{B\subseteq X; \Invobr fB\in\FF\}$ is a filter on the set $X$ called the image of $\FF$ under $f$.

We say that a filter $\FF$ on a topological space $X$ converges to a point $x\in X$ if it contains the filter of all neighborhoods of the point $x$, i. e., $\mc N(x)\subseteq \FF$.

\begin{DEF}
Let $\I$ be an ideal on a set $S$. A map $\Zobr fSX$ is $\I$-convergent to $x\in X$ if
$$\Invobr fU\in\FI$$
is true for each neighborhood $U$ of the point $x$. Equivalently we can say that the filter $\Obr f{\FI}$ converges to $x$.
\end{DEF}

\begin{REM}
As claimed in the above definition $\I$-convergence is equivalent to convergence of the filter $\Obr f{\FI}$. Similar result is true for many other notions, such as $\I$-cluster points, $\I$-Cauchy sequences. So when dealing with some notion related to $\I$-convergence we often can formulate this notion using the dual filter $\FI$ on a set $S$ and also using the filter $\Obr f{\FI}$ on the space $X$.
The main reason we are working with the ideals and not the filters is that this is continuation of some research done previously in \cite{KSW,DASGHOSAL,MACSLEIK} and many other papers; where the results were formulated using ideals. (Of course, every such result can be easily reformulated using dual filters.)
\end{REM}

Let us mention some special cases of $\I$-convergence.

Let $X$ be a topological space and $\FF$ be any filter on $X$. If $\I:=\{X\sm F; F\in\FF\}$ is the dual ideal to $\FF$, then the map $\Zobr{id_X}XX$ is $\I$-convergent to $x$ if and only if $\FF$ is convergent to $x$.

If $(D,\le)$ is a directed set, then we can consider the filter $\FD$ on the set $D$ generated by the sets of the form $\{x\in D; x\ge d\}$, where $d\in D$. We denote the dual ideal to $\FD$ as $\ID$. If $\Zobr fDX$ is a net in a topological space, then the convergence of the net $(f(d))_{d\in D}$ to a point $x$ is equivalent to the $\ID$-convergence of the function $f$ to $x$.  The filter $\FD$ is called
\emph{section filter} by some authors (e.g. \cite[p.60]{BOURBAKIGTENG} and \cite[p.35, Section 2.6]{ALIPRANTISBORDER}).

This shows that the notion of $\I$-convergence generalizes the two types of convergence that are nowadays most commonly used in general topology.
It is perhaps worth mentioning that Bourbaki used the dual definition (based on filters) as a common setting for various types of convergence, see
\cite{BOURBAKIGTENG,DIXMIERGT}.

Another type of $\I$-convergence that was widely studied is the statistical convergence. This is the $\I$-convergence for the ideal
$\Id=\{A\subseteq\NN; d(A)=0\}$ consisting of all sets having asymptotic density zero.

As we have already mentioned in the introduction, the properties of statistical convergence were one of the motivations for studying $\I$-convergence and $\I^*$-convergence. In this paper we will need some facts about $\IhJ$-convergence. This type of convergence was defined in \cite{MACSLEIK} and it generalizes $\I^*$-convergence.

\begin{DEF}
Let $\I$, $\J$ be ideals on a set $S$. Let $X$ be a topological space and $x\in X$. A function $\Zobr fSX$ is said to be
\emph{$\IhJ$-convergent} to $x$ if there is a set $M\in\FI$ such that the function $\Zobr gSX$ given by
$$g(s)=
  \begin{cases}
    f(s) & s\in M, \\
    x & s\notin M,
  \end{cases}
$$
is $\J$-convergent to $x$.
\end{DEF}

Notice that $\IhJ$-convergence can be equivalently defined by saying that $f|_M$ is $\J|_M$-convergent to $x$ for some $M\in\FI$.

For $S=\NN$ and $\J=\Fin$ we get the notion of \emph{$\I^*$-convergence} of sequences, which was studied, for example, in \cite{KSW,KMSS,LAHDASTOP}.

When studying the relationship between $\I$-convergence and $\IhJ$-convergence, the following condition was important.
\begin{DEF}
Let $\I$, $\J$ be ideals on the same set $S$. We say that the condition $\APIJ$ is fulfilled or that the ideal $\I$ has the \emph{additive property with respect to $\J$} if, for every sequence of pairwise disjoint sets $A_n\in\I$, there exists a sequence $B_n\in\I$ such that $A_n\triangle B_n\in\J$ for each $n$ and $\bigcup_{n\in\NN} B_n\in\I$.
\end{DEF}
Several equivalent reformulations of $\APIJ$ are given in \cite[Lemma 3.9]{MACSLEIK}.

In \cite[Theorem 3.11]{MACSLEIK} it was shown that if $\APIJ$ holds then $\I$-convergence implies $\IhJ$-convergence. If the space $X$ is countably generated and it is not discrete, then the condition $\APIJ$ is not only sufficient but also necessary.

In the case $S=\NN$ and $\J=\Fin$ we get the condition AP, which was used in \cite{KSW} to characterize ideals for which $\I$-convergence and $\I^*$-convergence of sequences in metric spaces is equivalent. Ideals fulfilling the condition $\APIFin$ are
sometimes called \emph{P-ideals} (see for example \cite{BADEKO}, \cite{FARAHMEMOIRS} or \cite{FILIPOWBOLZ}).

The concept of Cauchyness plays the central role in this paper. To study the properties of Cauchy functions in natural framework we will use uniform spaces. We give the definition of uniform space following Engelking's book \cite{ENGNEW}.
\begin{DEF}
Let $X$ be a set. A set $U\subseteq X\times X$ is called \emph{entourage of the diagonal} if $\Delta=\{(x,x); x\in X\}\subseteq U$ and $U=\inv U$.

Let $\Phi$ be a family of entourages of diagonal. The pair $(X,\Phi)$ is called \emph{uniform space} if for all entourages $U$, $V$ the conditions are satisfied:
\begin{compactenum}
\enu
  \item $U\in\Phi$ $\land$ $V\supseteq U$ $\Ra$ $V\in\Phi$;
  \item $U,V\in\Phi$ $\Ra$ $U\cap V\in\Phi$;
  \item $U\in\Phi$ $\Ra$ $(\exists V\in\Phi) V\circ V=\{(x,z); (\exists y\in X) (x,y),(y,z)\in V\} \subseteq U$;
  \item $\bigcap\Phi=\Delta$.
\end{compactenum}

The sets of the form
$$U[x]=\{y\in X; (x,y)\in U\}$$
give a local base at $x$ for the topology induced by the uniformity $\Phi$.
\end{DEF}

Note that the last condition in the definition of uniformity implies that the induced topology is Hausdorff. (Some authors omit this last condition in the definition of uniformity. Similarly, some text do not require entourages to be symmetric, the condition $U\in\Phi$ $\Ra$ $\inv U\in\Phi$ is included in the definition instead.)

A uniform space $(X,\Phi)$ is called \emph{complete} if every Cauchy net in $X$ is convergent. If $X$ is a metric space, then it suffices to require this condition for sequences. A complete uniform space can be equivalently defined as a uniform space in which every Cauchy filter is convergent. Another equivalent characterization of completeness is that every family $\mc F$ of closed subsets of $X$, which has finite intersection property and contains arbitrarily small sets, has a non-empty intersection. A family $\mc F$ of subsets of a uniform space $(X,\Phi)$ is said to contain \emph{arbitrarily small sets} if, for every $U\in\Phi$, there is an $F\in\mc F$ with $F\times F\subseteq U$. See \cite[p.446, Theorem 8.3.20, Theorem 8.3.21]{ENGNEW} for more details.

Every uniform space $(X,\Phi)$ has a \emph{completion} $(\widetilde X,\widetilde \Phi)$. The completion is unique up to isomorphism and it has the same weight as the original uniformity. See \cite[Theorem 8.3.12]{ENGNEW}.

\section{$\IhJ$-convergent and $\IhJ$-Cauchy functions}

Now we can define in a full generality the notion of Cauchy function and make some basic observations.

\begin{DEF}
Let $S$ be a set and $(X,\Phi)$ be a uniform space. Let $\I$ be an ideal on the set $S$. Let $\Zobr fSX$ be a map. The map $f$ is called \emph{$\I$-Cauchy} if for any $U\in\Phi$ there exists an $m\in S$ such that
$$\{n\in S; (f(n),f(m))\notin U\}\in\I.$$
\end{DEF}

\begin{LM}\label{LMIFFCAUCHY}
Let $(X,\Phi)$ be a uniform space and let $\I$ be an ideal on a set $S$. For a function $\Zobr fSX$ following are equivalent.
\begin{compactenum}
\enu
	\item $f$ is $\mathcal{I}$-Cauchy.
	\item For any $U\in\Phi$ there is $m\in S$ such that
			$$\{n\in S; (f(n),f(m))\in U\}\in\FI.$$
	\item For every $U\in\Phi$ there exists a set $A \in \mathcal{I}$ such that $m, n \notin A$ implies $(f(m), f(n)) \in U$.
	\item The filter $\Obr f{\FI}$ is a Cauchy filter.
\end{compactenum}
\end{LM}

The proof is straightforward and so it is omitted. In the case of sequences equivalence of some of these conditions was shown in \cite[Proposition 4]{DEMSICAUCH}.

It can be easily seen that every $\I$-convergent function is $\I$-Cauchy.

Again we can mention some special cases of this notion. A filter $\FF$ on a uniform space $X$ is Cauchy if and only if $\Zobr {id_X}XX$ is $\I$-Cauchy with respect to the dual ideal. A net $\Zobr fDX$ on a directed set $D$ is Cauchy if and only if it is $\ID$-Cauchy.

It is easy to get directly from the definition that $\I$-Cauchy function is $\J$-Cauchy for any finer ideal:
\begin{LM}\label{LMFINERCAUCHY}
Let $\I_1$, $\I_2$ be ideals on a set $S$ such that $\I_1\subseteq\I_2$. Let $X$ be a uniform space. If $\Zobr fSX$ is $\I_1$-Cauchy then it is also $\I_2$-Cauchy.
\end{LM}

This can also be deduced from the fact that filter, which is finer than a Cauchy filter, is again Cauchy, see \cite[Proposition 2.3.2]{BEATTIEBUTZMANNCONVERGENCEFA}, \cite[p.188]{BOURBAKIGTENG}, \cite[p.299]{GAHLER} or \cite[19.3]{SCHECHTERHBK}.

\begin{DEF}
Let $S$ be a set and $(X,\Phi)$ be a uniform space. Let $\I$, $\J$ be ideals on the set $S$.  A map $\Zobr fSX$ is said to be $\IhJ$-Cauchy if there is a subset $M\subseteq S$ such that $M\in\FI$ and the function $f|_M$ is $\J|_M$-Cauchy.
\end{DEF}

In the case that $\J=\Fin$ we obtain the notion of $\I^*$-Cauchy sequences, which was studied in \cite{NABPEHGUR}.
The $\I^*$-Cauchy nets introduced in \cite{DASGHOSALCOMPLETE} are precisely the $\IhJh{\I}{\ID}$-Cauchy functions.

It is relatively easy to see directly from definition that every $\IhJ$-convergent function is $\IhJ$-Cauchy.

\begin{LM}\label{LMFINERIKCAUCHY}
Let $S$ be a set and $(X,\Phi)$ be a uniform space. Let $\I$, $\I_1$, $\J$, $\J_1$ be ideals on the set $S$ such that $\I\subseteq\I_1$ and $\J\subseteq\J_1$.

If a map $\Zobr fSX$ is $\IhJ$-Cauchy, then it is also $\IhJh{\I_1}{\J}$-Cauchy.

If a map $\Zobr fSX$ is $\IhJ$-Cauchy, then it is also $\IhJh{\I}{\J_1}$-Cauchy.
\end{LM}

\begin{proof}
If $\Zobr fSX$ is $\IhJ$-Cauchy then there is a subset $M\in\FI$ such that $f|_M$ is $\J|_M$-Cauchy. Since $\FI\subseteq\FF(\I_1)$, we have $M\in\FF(\I_1)$. This means that $f$ is also $\IhJh{\I_1}{\J}$-Cauchy.

To prove the second part we just need to notice that $\J\subseteq\J_1$ implies $\J|_M\subseteq\J_1|_M$. From Lemma \ref{LMFINERCAUCHY} we get that if $f|_M$ is $\J|_M$-Cauchy, then it is also $\J_1|_M$-Cauchy. This proves that $f$ is $\IhJh{\I}{\J_1}$-Cauchy.
\end{proof}

\begin{LM}\label{LMKCAUCHY}
Let $S$ be a set and $(X,\Phi)$ be a uniform space. Let $\I$, $\J$ be ideals on $S$. If a map $\Zobr fSX$ is $\J$-Cauchy, then it is also $\IhJ$-Cauchy.
\end{LM}

\begin{proof}
If we take $M=S$, then $M\in\FI$. In this case $\J|_M=\J$, hence $f$ is $\J|_M$-Cauchy. This shows that $f$ is $\IhJ$-Cauchy.
\end{proof}

\subsection{$\IhJ$-convergent and $\IhJ$-Cauchy functions for $\I=\J$}

In this part we discuss special cases of the notions studied in this paper in the case $\I=\J$.

\begin{PROP}\label{PROPI=J}
Let $\Zobr fSX$ be a map, $\I$ be an ideal on the set $S$ and $X$ be a topological space.

\begin{compactenum}
\enu
  \item The map $f$ is $\IhI$-convergent to $x$ if and only if $f$ is $\I$-convergent to $x$.
  \item Let as assume that $(X,\Phi)$ is additionally a uniform space. Then $f$ is $\IhI$-Cauchy if and only if it is $\I$-Cauchy.
\end{compactenum}
\end{PROP}

\begin{proof}
(i): $\boxed{\Ra}$ Suppose that $f$ is $\IhI$-convergent to $x$. Hence there is a set $M\in\FI$ such that $f|_M$ is $\I|_M$ convergent. This means that for any neighborhood $U$ of $x$, there exists $F\in\FI$ such that
$$\Invobr fU\cap M=F\cap M.$$
Clearly, this implies that $\Invobr fU\in\FI$, since $\Invobr fU\supseteq F\cap M$ and $F\cap M\in\FI$.

$\boxed{\Leftarrow}$ Suppose that $f$ is $\I$-convergent to $x$. Then for $M=S$ we have $\I|_M=\I$. This implies that $f$ is also $\IhI$-convergent to $x$.

(ii): $\boxed{\Ra}$ Let $f$ be $\IhI$-Cauchy. This means that, for some $M\in\FI$, the restriction $f|_M$ is $\I|_M$-Cauchy, i.e., for each $U\in\Phi$ there is $F\in\FI$ such that
$$x,y\in F\cap M \qquad \Ra \qquad (f(x),f(y))\in U.$$
Since the set $F\cap M$ belongs to $\FI$, this shows that $f$ is also $\I$-Cauchy.

$\boxed{\Leftarrow}$ Again, we can take $M=S$. We have $\I|_M=\I$, so if $f$ is $\I$-Cauchy, then $f|_M=f$ is also $\I|_M$-Cauchy. Thus $f$ is $\IhI$-Cauchy.
\end{proof}

If we combine Proposition \ref{PROPI=J} with results about finer ideals, we get the following result:
\begin{COR}\label{CORSUBSETIMPLIES}
Let $\I$, $\J$ be ideals on a set $S$. Let $\Zobr fSX$ be a map from $S$ to a topological space (a uniform space) $X$

Suppose that $\J\subseteq\I$. Then if $f$ is $\IhJ$-convergent to $x$, it is also $\I$-convergent to $X$. Similarly,
if $f$ is $\IhJ$-Cauchy then it is also $\I$-Cauchy.
\end{COR}

\begin{proof}
Let $\J\subseteq\I$.

By \cite[Proposition 3.6]{MACSLEIK} every $\IhJ$-convergent function is also $\IhI$-convergent. This is the same thing as $\I$-convergent.

By Lemma \ref{LMFINERIKCAUCHY} every $\IhJ$-Cauchy function $f$ is also $\IhI$-Cauchy. The latter is equivalent to $f$ being $\I$-Cauchy.
\end{proof}

Of course, both results in this corollary could be also easily proved directly from the definition.

The first part of this Corollary is \cite[Proposition 3.7]{MACSLEIK}. As a special case of the second part for $\J=\Fin$ we obtain \cite[Theorem 3]{NABPEHGUR}.

\subsection{$\IhJ$-convergence and $\IhJh{(\I\vee\J)}{\J}$-convergence}

For any two ideals $\I$, $\J$ we have the ideal
$$\I\vee\J=\{A\cup B; A\in\I, B\in\J\}.$$
This is the smallest ideal containing both $\I$ and $\J$.

The dual filter is
$$\FF(\I\vee\J)=\FI\vee\FJ=\{F\cap G; F\in\FI, G\in\FJ\}.$$

Clearly, if $M\in \FF(\I)$, then also $M\in\FF(\I\vee\J)$. Therefore any $\IhJ$-convergent function is also $\IhJh{(\I\vee\J)}{\J}$-convergent, every $\IhJ$-Cauchy function is $\IhJh{(\I\vee\J)}{\J}$-Cauchy. (These properties only depend on the ideal $\J|_M$.) These facts are also clear from Lemma \ref{LMFINERIKCAUCHY}, since $\I\subseteq\I\vee\J$.

The question whether the opposite implication is true is answered in the following proposition:
\begin{PROP}\label{PROPIVEEJ}
Let $\I$, $\J$ be ideals on a set $S$.

A function $\Zobr fSX$ is $\IhJ$-convergent to $x$ if and only if it is $\IhJh{(\I\vee\J)}{\J}$-convergent to $x$.

A function $\Zobr fSX$ is $\IhJ$-Cauchy if and only if it is $\IhJh{(\I\vee\J)}{\J}$-Cauchy.
\end{PROP}

\begin{proof}
We already have one of the implications in both cases.

\emph{Convergence.} Suppose that $f$ is $\IhJh{(\I\vee\J)}{\J}$-convergent to $x$. This means that there is a set $M\in\FF(\I\vee\J)$ such that $f|_M$ is $\J|_M$-convergent, i.e., for every neighborhood $U$ of the point $x$ there exists $F\in\FJ$ such that
$$\Invobr fU\cap M=F\cap M.$$

Since $M\in\FF(\I\vee\J)$, it has form
$$M=M'\cap G$$
for some $M'\in\FI$ and $G\in\FJ$.

Now we have
$$\Invobr fU\cap M'\supseteq \Invobr fU\cap M = (F\cap G)\cap M'.$$
Since $F\cap G\in\FJ$, this shows that $\Invobr fU\cap M'\in\FF(\J|_{M'})$ and that $f$ is $\IhJ$-convergent.

\emph{Cauchy condition.} Now assume that $f$ is $\IhJh{(\I\vee\J)}{\J}$-Cauchy. This means that there is $M\in\FF(\I\vee\J)$ such that $f|_M$ is $\J|_M$-Cauchy. I.e., for every $U\in\Phi$ there exists $F\in\FJ$ such that
$$\{(f(x),f(y)); x,y\in F\cap M\}\subseteq U.$$
Again, we know that
$$M=M'\cap G$$
for some $M'\in\FI$ and $G\in\FJ$.

If we put $F'=F\cap G$ then $F'\in\FJ$ and
$$F'\cap M'=F\cap G\cap M'=F\cap M$$
which means that
$$\{(f(x),f(y)); x,y\in F'\cap M'\}\subseteq U.$$
This shows that $f$ is $\IhJ$-Cauchy.
\end{proof}

\begin{REM}
Proposition \ref{PROPIVEEJ} shows that we do not lose any generality, if we work with the assumption $\J\subseteq\I$. (If we are given any two ideals on a set $S$, then we can modify them to ideals such that the notion of $\IhJ$-convergence is the same for both pairs of ideals and, additionally, $\J\subseteq\I$.)

Moreover, in most settings $\J\subseteq\I$ is a very natural condition to assume.
\end{REM}

If $\I\subseteq\J$, then $\I\vee\J=\J$. So in this case we immediately get from Propositions \ref{PROPIVEEJ} and \ref{PROPI=J} that
$\IhJ$-convergence is equivalent to $\J$-convergence; and a function is $\IhJ$-Cauchy if and only it is $\J$-Cauchy.

\section{$\IhJ$-Cauchy functions and completeness}

Since we have defined some kind of Cauchy condition, it is very natural to ask how it is related to completeness of the space $X$. First we show that if $X$ is complete, then every $\IhJ$-Cauchy function is also $\IhJ$-convergent.

\begin{LM}\label{LMCOMPLCONV}
If $(X,\Phi)$ is a complete uniform space then any $\I$-Cauchy function $\Zobr fSX$ is also $\I$-convergent.
\end{LM}

\begin{proof}
Recall that $f$ is $\I$-Cauchy if and only if $\Obr f{\FI}$ is Cauchy (Lemma \ref{LMIFFCAUCHY}) and $f$ is $\I$-convergent if and only if $\Obr f{\FI}$ is convergent. So now it suffices to use the fact that in a complete uniform space every Cauchy filter is convergent \cite[Theorem 8.3.21]{ENGNEW}.
\end{proof}

Notice that this gives a simpler proof of one implication of \cite[Theorem 2]{DEMSICAUCH}. Also \cite[Theorem 7]{DASGHOSALCAUCHYNETS} is a special case of Lemma \ref{LMCOMPLCONV}. In fact, this gives a complete answer to the problems posed in \cite[Remark 1]{DASGHOSALCAUCHYNETS}.

As another direct consequence of this we get the well-known result that Cauchy nets in complete uniform spaces are convergent, by applying the above lemma to the ideal $\ID$.

From Lemma \ref{LMCOMPLCONV} we get immediately:
\begin{COR}\label{CORCOMPLCONV}
If $(X,\Phi)$ is a complete uniform space, then any $\IhJ$-Cauchy function $\Zobr fSX$ is $\IhJ$-convergent.
\end{COR}

\begin{proof}
If $f$ is $\IhJ$-Cauchy, then there exists $M\in\FI$ such that $f|_M$ is $\J|_M$-Cauchy. From Lemma \ref{LMCOMPLCONV} we get that $f|_M$ is $\J|_M$-convergent and, consequently, $f$ is $\IhJ$-convergent.
\end{proof}

Once we know that in a complete space a function is $\I$-convergent if and only if it is $\I$-Cauchy and it is $\IhJ$-convergent if and only if it is $\IhJ$-Cauchy, some results about $\IhJ$-Cauchy functions can be shown using results on $\IhJ$-convergence for the completion of the given space.
This is true for some results given in the preceding section.
We will also show Theorems \ref{THMAPIJIMPLIES} and \ref{THMIMPLIESAPIJ} in this way.

\subsection{One ideal for each directed set}

We already know that in complete spaces $\IhJ$-Cauchy function must be $\IhJ$-convergent. A natural question is whether this condition characterizes complete spaces.

Of course, if we require this condition for all ideals $\I$ and $\J$ then we get that it holds for all nets in $X$ and, consequently, $X$ is complete. The question is whether we can somehow restrict the class of ideals for which this implication holds in a such way, that this still implies completeness of $X$.

The arguments similar to the following lemma have been used in the proof of \cite[Theorem 5]{DASGHOSALCAUCHYNETS}, which gives a characterization of complete uniform spaces using $\I$-Cauchy nets.

\begin{LM}\label{LMCOFINAL2}
Let $x=(x_d)_{d\in D}$ be a Cauchy net in a uniform space $(X,\Phi)$ and let $l\in X$.
Suppose that for every $U\in\Phi$ the set
$$\Invobr x{U[l]}=\{d\in D; (x_d,l)\in U\}$$
is cofinal in $D$.

Then the net $x$ converges to $l$.
\end{LM}

\begin{proof}
Let $U\in\Phi$ and let $V\in\Phi$ be such $V\circ V\subseteq U$. We want to show that the net $(x_d)_{d\in D}$ is eventually in the neighborhood $U[l]$.

Since $x$ is a Cauchy net,
there exists $d_0\in M$ such that
$$(x_d,x_e)\in V$$
for any $d,e\ge d_0$.

Since $\Invobr x{V[l]}$ is cofinal in $D$,
there exists $d'\in D$ such that $d'\ge d_0$ and
$$(x_{d'},l)\in V.$$

From this we get for any $d\ge d_0$ that $(x_d,x_{d'})\in V$ and $(x_{d'},l)\in V$.
Together we get
$$(x_d,l)\in U$$
whenever $d\ge d_0$.

This shows that $(x_d)_{d\in D}$ converges to $l$.
\end{proof}

\begin{PROP}
Let $(D,\le)$ be a directed set and $(X,\Phi)$ be a uniform space.

If there exists a $D$-admissible ideal $\I$ on the set $D$ with the property that every $\I$-Cauchy net in $X$ is $\I$-convergent, then every Cauchy net $(x_d)_{d\in D}$ on the directed set $D$ with values in $X$ is also convergent.
\end{PROP}

If $D$ is a directed set, then an ideal $\I$ on the set $D$ is called \emph{$D$-admissible} if $\I$ is a proper ideal and $\ID\subseteq\I$. This notion was introduced in \cite{DASGHOSALCAUCHYNETS}.

Notice that a proper ideal $\I$ is $D$-admissible if and only if $\FD\subseteq\FI$. In particular, this implies that $F\cap M\ne\emps$ for any $F\in\FD$, $M\in\FI$. This means that every set $M\in\FI$ is cofinal in $D$.

\begin{proof}
Let $\Zobr xDX$ be a Cauchy net. This means that $x$ is $\ID$-Cauchy and we get from Lemma \ref{LMFINERCAUCHY} that $x$ is also $\I$-Cauchy.

Therefore the net $x$ is $\I$-convergent to some limit $l\in X$. Thus for any neighborhood $U[l]$ of $l$ we get that
$$\Invobr x{U[l]}\in\FI,$$
which implies that $\Invobr x{U[l]}$ is cofinal in $D$.

Now Lemma \ref{LMCOFINAL2} implies that the net $(x_d)_{d\in D}$ converges to $l$.
\end{proof}

From this we get the following result:
\begin{COR}\label{CORICAUCHYCOMPLETE}
Let $(X,\Phi)$ be a uniform space. Suppose that for every directed set $D$ there exists a $D$-admissible $\I$ ideal such that every $\I$-Cauchy net in $X$ is $\I$-convergent. Then $X$ is complete.
\end{COR}
This results was shown before in \cite[Theorem 5]{DASGHOSALCAUCHYNETS}.

For metric spaces we get
\begin{COR}
Let $X$ be a metric space and $\I$ be an admissible ideal on the set $\NN$. If every $\I$-Cauchy sequence is $\I$-convergent, then $X$ is complete.
\end{COR}

In particular, for the ideal $\Id$ this means:
\begin{COR}
If $X$ is a metric space such that every statistically Cauchy sequence is statistically convergent, then $X$ is complete.
\end{COR}

This answers \cite[Problem 2.16]{DIMAIOKOCINACSTAT}: If we take any metric space which is not complete, then there is a sequence in $X$ which is statistically Cauchy but not statistically convergent. (Of course, no such sequence can exist in a complete space, see Lemma \ref{LMCOMPLCONV}.)
\begin{PROP}
Let $(D,\le)$ be a directed set and $(X,\Phi)$ be a uniform space. Let $\I$, $\J$ be $D$-admissible ideals such that $\J\subseteq\I$.

Suppose that every $\IhJ$-Cauchy net $\Zobr xDX$ is $\IhJ$-convergent. Then also every Cauchy net $\Zobr xDX$ on the directed set $D$ is convergent.
\end{PROP}

\begin{proof}
Let $\Zobr xDX$ be a Cauchy net. Since $\ID\subseteq\J$, it is also $\J$-Cauchy (Lemma \ref{LMFINERCAUCHY}) and, consequently, it is $\IhJ$-Cauchy (Lemma \ref{LMKCAUCHY}).

According to our assumptions, the net $x$ is then also $\IhJ$-convergent to some point $l\in X$. This means that there is a set $M\in\FI$ such that $x|_M$ is $\J|_M$-convergent to $l$. That is, for every neighborhood $U[l]$ of $l$ we have
$$\Invobr x{U[l]}=F\cap M$$
for some $F\in\FJ$.

Since $\FJ\subseteq\FI$, we get that $F\cap M\in\FI$ and consequently, $F\cap M$ is cofinal in $D$. Now the claim follows from Lemma \ref{LMCOFINAL2}.
\end{proof}

Again, we can get a result similar to Corollary \ref{CORICAUCHYCOMPLETE}:
\begin{COR}\label{CORIKCAUCHYCOMPLETE}
Let $(X,\Phi)$ be a uniform space. Suppose that for every directed set $D$ there exist $D$-admissible ideals $\I$ and $\J$ such that every $\IhJ$-Cauchy net in $X$ is $\IhJ$-convergent. Then $X$ is complete.

If $X$ is a metric space, it suffices to have such ideals for $D=\NN$.
\end{COR}

\subsection{One ideal for each cardinality}

It is a natural question whether in uniform spaces we need to have the implication \uv{Cauchy net $\Ra$ convergent net}, for all directed sets or this can be restricted to some class of directed sets. (In the case of metric space only one directed set $(\mathbb N,\le)$ was sufficient.)

Recall that a \emph{base for the uniformity $\Phi$} is a family $\mc B\subset\Phi$ such that for every $V\in\Phi$ there exists a
$W\in\mc B$ such that $W\subseteq V$. The smallest cardinality of the base is called the \emph{weight of the uniformity $\Phi$} and is denoted by $w(\Phi)$.

By $[\vk]^{<\omega}$ we denote the set of all finite subsets of $\vk$. Note that $([\vk]^{<\omega},\subseteq)$ is a directed set.

\begin{PROP}
Let $(X,\Phi)$ be a uniform space and $w(\Phi)=\vk$. If every net on the directed set $([\vk]^{<\omega},\subseteq)$ is convergent, then $(X,\Phi)$ is complete.
\end{PROP}

\begin{proof}
Let $\mc B=\{U_\alpha; \alpha<\vk\}$ be a base for the uniformity $\Phi$. Let $(x_d)_{d\in D}$ be a Cauchy net in $X$.

For every $\alpha$ there exists $d_\alpha\in D$ such that
$$d,e\ge d_\alpha \qquad\Ra\qquad (x_d,x_e)\in U_\alpha.$$
For any $F\in[\vk]^{<\omega}$ we can choose $d_F$ such that $d_F\ge d_\gamma$ for each $\gamma\in F$.
By setting $x_F=x_{d_F}$ we get a net on the directed set $[\vk]^{<\omega}$.

We claim that the net $(x_F)$ is Cauchy. Indeed, if we chose any $U_\alpha$ and any $F\in [\vk]^{<\omega}$ such that $\alpha\in F$, then
$$(x_G,x_H)\in U_\alpha$$
for any $G,H\supseteq F$.

Then the net $(x_F)$ is also convergent. Let us denote the limit by $l$. We will show that the net $(x_d)_{d\in D}$ converges to $l$, as well.

Let $U\in\Phi$ and let $U_\alpha\in\mc B$ be such that $U_\alpha\circ U_\alpha\subseteq U$.

Then there is a set $F\in[\vk]^{<\omega}$ such that for $G\in[\vk]^{<\omega}$, $G\supseteq F$ we have
$$(x_G,l) = (x_{d_G},l) \in U_\alpha.$$
If we moreover assume that $\alpha\in F$, then we get $d_G>d_\alpha$ and
$$(x_e,x_{d_G})\in U_\alpha$$
for any $e\ge d_\alpha$. Together we get
$$(x_e,l) \in U_\alpha\circ U_\alpha \subseteq V$$
for each $e\ge d_\alpha$.

This proves the convergence of the net $(x_d)_{d\in D}$.
\end{proof}

Using the above proposition we can get results corresponding to Corollaries \ref{CORICAUCHYCOMPLETE} and \ref{CORIKCAUCHYCOMPLETE} where we do not work with all directed sets, but only with the directed sets of the form $([\vk]^{<\omega},\subseteq)$.

\section{Condition $\APIJ$}

In \cite{MACSLEIK} the relationship between the condition $\APIJ$ and the equivalence of $\I$-convergence and $\IhJ$-convergence was studied. (As a natural generalization of results from \cite{KSW} obtained for $\I$-convergence and $\I^*$-convergence of sequences.) These results were obtained only in the case that we are working in a first countable topological space.

We want to prove similar results for the notion of $\I$-Cauchy and $\IhJ$-Cauchy functions. Again, these results are not true for arbitrary uniform spaces, as can be seen from the examples included below. So in the rest of the paper we will work mostly with metric spaces. (These examples also show that \cite[Theorem 5, Theorem 6]{DASGHOSALCOMPLETE} are not valid in arbitrary uniform spaces.)

For the following two theorems we have given two proofs. One of them uses completion of a metric space and some results on $\IhJ$-convergence obtained in \cite{MACSLEIK}. The other proof is self-contained.

\begin{THM}\label{THMAPIJIMPLIES}
Let $X$ be a metric space and $\I$, $\J$ be ideals on a set $S$ such that the condition $\APIJ$ holds. If $\Zobr fSX$ is $\I$-Cauchy, then it also is $\IhJ$-Cauchy.
\end{THM}

\begin{proof}
The metric space $X$ has a completion $\widetilde X$. Let us consider the function $f$ as the function from $S$ to $\widetilde X$.

Since $\widetilde X$ is complete, the function $f$ has an $\I$-limit. Since the ideals $\I$, $\J$ fulfill the condition $\APIJ$, every $\I$-convergent function in $\widetilde X$ is $\IhJ$-convergent according \cite[Theorem 3.11]{MACSLEIK}. Consequently, the function $f$ is $\IhJ$-Cauchy in $\widetilde X$. Of course, this implies that it is also $\IhJ$-Cauchy in $X$.
\end{proof}

\begin{proof}
Since $\Zobr fSX$ is $\I$-Cauchy, so for every $r \in \mathbb{N}$, we can find a set $G_r \in \I$ such that $s, t \notin G_r$ implies $d (f(s), f(t)) < \frac{1}{r}$. Let $A_1 = G_1$, $A_2 = G_2 \setminus G_1$, $A_3 = G_3 \setminus G_2$, etc. Then $\{A_i; i\in\NN\}$ is a countable family of mutually disjoint sets in $\I$. Since $\APIJ$ holds, there exists a family of sets $\{B_i; i\in\NN\}$ belonging to $\I$ such that $B=\bigcup\limits_{i\in\NN} B_i\in\I$ and for every $j \in \mathbb{N}$ we have $A_j \triangle B_j \in \J$, i.e., $A_j \triangle B_j = S \setminus C_j$ for some $C_j \in \FJ$.

Let $M=S\sm B$.

Let $\ve > 0$ be given.
Choose $r \in \mathbb{N}$ such that $\frac{1}{r} < \ve$. Now
$$G_r \cap M = G_r \setminus B \subseteq \bigcup\limits^r_{i = 1}(A_i \setminus B_i) \subseteq \bigcup\limits^r_{i = 1}(S \setminus C_i) = S \setminus C$$
where $C = \bigcap\limits^r_{i = 1}C_i \in \FJ$. Then $G_r^c \cap M \supseteq M \cap C$ (for otherwise there is an $s \in M \cap C$ but $s \notin G_r^c$ and so $s \in G_r \cap M \subseteq S \setminus C$ which is a contradiction). This shows that $s, t \in C \cap M \Rightarrow s, t \in G_r^c \Rightarrow d(f(s), f(t)) < \frac{1}{r} < \ve$. But $C \cap M \in \FF(\J|_M)$ which implies that $f|_M$ is $\J|_M$-Cauchy, i.e., $f$ is $\IhJ$-Cauchy.
\end{proof}

If we take $S=\NN$ and $\J=\Fin$ in Theorem \ref{THMAPIJIMPLIES}, we get \cite[Theorem 4]{NABPEHGUR}.
Also in view of the fact that a uniform space, which has a countable base of the uniformity, is metrizable, we get \cite[Theorem 5]{DASGHOSALCOMPLETE} as a special case.

\begin{THM}\label{THMIMPLIESAPIJ}
Let $\I$, $\J$ be ideals on a set $S$ and $X$ be a metric space which is not discrete. Suppose that every $\I$-Cauchy function is $\IhJ$-Cauchy. Then the condition $\APIJ$ holds.
\end{THM}

We will use the following result, which follows from \cite[Theorem 3.12, Remark 3.13]{MACSLEIK}:
Suppose that $X$ is a Hausdorff space and $x\in X$ is a non-isolated point. Suppose that every function $\Zobr fSX$, which is $\IhJ$-convergent to $x$ is also $\I$-convergent to $x$. Then the ideals $\I$ and $\J$ fulfill the condition $\APIJ$.

\begin{proof}
Let $x$ be a non-isolated point of $X$. According to \cite[Theorem 3.12, Remark 3.13]{MACSLEIK} it suffices to show that every function $\Zobr fSX$, which is $\I$-convergent to $x$ is also $\IhJ$-convergent to $x$.

So let us assume that $f$ is $\I$-convergent to $x$. Then $f$ is also $\I$-Cauchy and consequently it is $\IhJ$-Cauchy.

If we consider the function $f$ as a function from $S$ to the completion $\widetilde X$, then it $\IhJ$-converges to some point of $\widetilde X$. We only need to show that it converges to $x$.

Suppose that $f$ is $\IhJ$-convergent to a point $y\in\widetilde X$ such that $y\ne x$. Since $\widetilde X$ is Hausdorff, we can choose neighborhoods $U\ni x$, $V\ni y$ such that $U\cap V=\emps$. Now we get that $\Invobr fU\in\FI$ and there exists an $M\in\FI$ such that $\Invobr fV\cap M\in\FF(\J|_M)$, which means that
$$\Invobr fV\cap M=G\cap M$$
for some $G\in\FJ$. We have
$$\Invobr fU \cap \Invobr fV\cap M=\Invobr fU \cap G\cap M = (\Invobr fU\cap M)\cap G=\emps.$$
Let us denote $M'=\Invobr fU\cap M$. The set $M'$ belongs to $\FI$ and, since $M'\subseteq S\sm G$, the set $M'$ belongs also to $\mc K$.

This means that $\J|_{M'}=\powerset{M'}$, i.e. it is not a proper ideal and every function from $M'$ to $X$ is $\J|_{M'}$-convergent to every point of $X$. In particular, $f$ is $\J|_{M'}$-convergent to $x$, which shows that $f$ is also $\IhJ$-convergent to $x$.
\end{proof}

\begin{proof}
Let $x_0$ be a non-isolated point in $X$. Let $(x_n)_{n\in\NN}$ be a sequence of distinct points in $X$ which is convergent to $x_0$. Let $\{A_i; i\in\NN\}$ be a sequence of mutually disjoint non-empty sets from $\I$. We define a function $\Zobr fSX$ by
$$
f(s) =
  \begin{cases}
    x_j & \text{if }s\in A_j, \\
    x_0 & \text{if }s\notin A_j\text{ for any }j.
  \end{cases}
$$
Let $\ve > 0$ be given. Choose $k \in \mathbb{N}$ such that $d(x_n,x_0) < \frac{\ve}{2}$ for each $n \geq k$. Now the set $D = \{s \in S; d(f(s),x_0) \geq \frac{\ve}{2}\}$ has the property that $D \subseteq A_1 \cup A_2 \cup \dots \cup A_k$ and so $D \in \I$ and $s, t \notin D$ implies $d(f(s), f(t)) \leq d(f(s), x_0) + d(f(t), x_0) < \ve$. This shows that $f$ is $\I$-Cauchy. Then by our assumption $f$ is also $\IhJ$-Cauchy. Hence there is a set $M \in \FI$ such that $f|_M$ is $\J|_M$-Cauchy. Let $B_j = A_j \setminus M$. Then $B_j \in \I$ for each $j$ and $\bigcup\limits_j B_j \subseteq S \setminus M \in \I$. Clearly $A_j \triangle B_j = A_j \setminus B_j = A_j \cap M$ for every $j$.

Case I. If $A_j \cap M \in \J$ for all $j \in \mathbb{N}$ then $\APIJ$ holds. If $A_j \cap M \notin \J$ for at most one $j$, say for $j_0$, then redefining $B_{j_0} = A_{j_0}$, $B_j = A_j \sm M$
when $ j \neq j_0$ we again observe that the condition $\APIJ$ holds.

Case II. Finally, if possible, suppose that $A_j \cap M \notin \J$ for at least two $j$'s. Let $l$ and $m$ be two such indices. We shall show that this is not possible. We have $C_1 = A_l \cap M \notin \J$ and $C_2 = A_m \cap M \notin \J$ where $l \neq m$. Now any $E \in \FF(\J|_M)$ is of the form $E = C \cap M$ where $C \in \FJ$. Now $C \cap C_1 \neq \emps$ for otherwise we will have $C_1 \subseteq S \setminus C \in \J$ which is not the case. By similar reasoning $C \cap C_2 \neq \emps$. Then there is an $s \in E$ such that $s \in A_l$ i.e. $f(s) = x_l$ and a $t \in E$ such that $t \in A_m$ i.e. $f(t) = x_m$. Now choosing $0 < \ve_0 < \frac{d(x_l,x_m)}{3}$ we observe that for every $E \in \FF(\J|_M)$ there exist points $s, t \in E$ such that $d(f(s), f(t)) > \ve_0$ or in other words for this $\ve_0 > 0$, there does not exist any $D \in \J|_M$ such that $s,t \notin D \Rightarrow d(f(s), f(t)) < \ve_0$. But this contradicts the fact that $f|_M$ is $\J|_M$-Cauchy. Hence Case II can not arise and this completes the proof of the theorem.
\end{proof}

As special cases of Theorem \ref{THMIMPLIESAPIJ} we get \cite[Theorem 1]{DASGHOSAL} (taking $S = \NN$ and $\J = \Fin$) and also \cite[Theorem 6]{DASGHOSALCOMPLETE} (for $S=D$ and $\J=\ID$).

\begin{figure}[h]
\begin{center}
\includegraphics{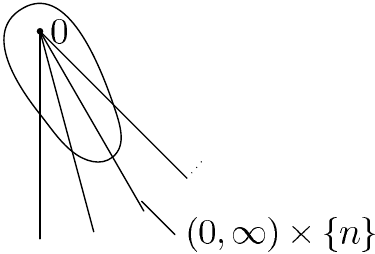}
\end{center}
\caption{The space from Example \ref{EXANIVEN} consists of countably many rays. The distance on each ray is stretched by the factor $\ve_n$. Typical neighborhood of $0$ looks like the set depicted in this figure.}\label{FIGNIVEN}
\end{figure}

Now we include an example of a uniform space $X$ such that there is a sequence in $X$ which is $\Id$-Cauchy but not $\Id^*$-Cauchy.
It is well known that the ideal $\Id$ is a P-ideal, i.e., the condition $\APIdFin$ holds; see, for example, \cite[Example 1.2.3(d)]{FARAHMEMOIRS}.
So the following example shows that Theorem \ref{THMAPIJIMPLIES} does not hold for arbitrary uniform spaces.
\begin{EXA}\label{EXANIVEN}
Let $X=\{0\}\cup\bigcup\limits_{n\in\NN} (0,\infty)\times\{n\}$. On this space we define the following family of pseudometrics. For every sequence $\ve=(\ve_n)$ of positive real numbers we define
\begin{align*}
d_\ve(0,(a,n))&=\ve_n a\\
d_\ve((a,n),(b,n))&=\ve_n \abs{b-a}\\
d_\ve((a,n),(b,m))&=\ve_n a+\ve_m b,&\text{for }n\ne m.
\end{align*}
This family of pseudometrics gives a uniformity $\Phi$ on $X$.

The space we obtain in this way is illustrated in Figure \ref{FIGNIVEN}.

We denote by $p_k$ the $k$-th prime. We decompose the set of positive integers as
$$C^{(k)}=\{n\in\NN; p_k\mid n+1, (\Dots p1{k-1},n+1)=1\}=: \{c^{(k)}_1<c^{(k)}_2<\dots <c^{(k)}_n<\dots\}.$$
Let $\Zobr f{\NN}X$ be defined as
$$f(c^{(k)}_n)=\left(\frac1{n+1},k\right).$$

We first show that this sequence is statistically convergent to $0$. Clearly, for any given neighborhood $U$ of $0$, the set $A(U):=\NN\sm\Invobr fU$ contains only finitely many elements from each $C^{(k)}$. This fact, together with the definition of $C^{(k)}$, implies that for any prime $p$ the set $\{n\in A(U); p\mid n+1\}$ is finite and, consequently, it has zero density. Now Corollary 1 of \cite{NIVENDENS} yields $d(A(U))=0$.

Now suppose that $f|_M$ is Cauchy sequence for some set $M$. Then the set $M$ intersects only finitely many of the sets $C^{(k)}$, $k\in\NN$. (Otherwise we would have a subsequence of the form $(a_k,n_k)$, where $n_k$ is an increasing sequence of positive integers, and by choosing $\ve_{n_k}=1/a_k$ we would get  $d_\ve((a_k,n_k),(a_l,n_l))>1$ for each $k$, $l$, contradicting the assumption that $f|_M$ is Cauchy.)
Therefore there exists $n$ such that $C^{(n)}\cap M = \emps$. As
$$d(C^{(n)})=\frac{(p_1-1)\dots(p_{n-1}-1)}{p_1p_2\dots p_n}$$
we have $\ol d(M)<1$ and $M\notin\FF(\Id)$.

Therefore the sequence $f$ is not $\I_d^*$-Cauchy.
\end{EXA}

The following example shows that Theorem \ref{THMIMPLIESAPIJ} is not true for arbitrary uniform spaces.
\begin{EXA}\label{EXAOMEGA1}
Let us recall that $\omone$ denotes the first uncountable ordinal with the usual ordering. Let $X$ be the topological space on the
set $\omone \cup \{\omone\}$ with the topology such that all points different from $\omone$ are isolated and the local base at the
point $\omega_1$ consists of all sets $U_\alpha=\{\beta\in X; \beta>\alpha\}$ for $\alpha<\omone$.

In \cite[Example 3.15]{MACSLEIK} it is shown that for any ideal $\I$ a sequence $\Zobr f{\NN}X$ is $\I$-convergent to $\omega_1$ if and only if there exists $M\in\FI$ such that $f(x)=\omega_1$ for each $x\in M$, i.e., $f|_M$ is a constant function. Since the remaining points of $X$ are isolated, the same is obviously true for all points $x\in X$. Therefore in this space the $\I$-convergence and $\I^*$-convergence of sequences are equivalent for any admissible ideal $\I$ on the set $\NN$.

We will show that the topology on $X$ can be obtained from a complete uniformity. Then we know that $\I$-Cauchy sequences in $X$ are precisely the $\I$-convergent sequences and $\I^*$-Cauchy sequences are precisely the $\I^*$-convergent sequences.
This means that taking any ideal $\I$ on $\NN$ which does not have the property $\APIFin$ gives us the desired counterexample.

It is relatively easy to see that if we take $B_\alpha=\{(x,x); x\in X\} \cup \{(x,y); x>\alpha, y>\alpha\}$ for $\alpha<\omone$, then $\{B_\alpha; \alpha<\omone\}$ is a base for a uniformity on the set $X$ and that the topology induced by this uniformity is precisely the topology described above.

Now let $\mc F$ be a system of closed sets which has finite intersection property and contains arbitrarily small sets. This means that for every $\alpha\in\omone$ there exists $F\in\mc F$ such that $F\cap\alpha=\emps$. Then we have $\omega_1\in F$ for each $F\in\mc F$. (If there is a closed set $F\in\mc F$ such that $\omega_1\notin F$, then we have $F\subseteq\alpha$ for some $\alpha<\omega_1$. If we take $F'$ such that $F'\cap\alpha=\emps$, we get that $F\cap F'=\emps$. This contradicts the finite intersection property.)

This shows that the uniformity described above is complete.
\end{EXA}

\begin{figure}[h]
\begin{center}
\includegraphics{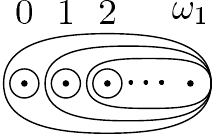}
\end{center}
\caption{The space $C(\omega_1)$ from Example \ref{EXAOMEGA1}}\label{FIGOMEGA1}
\end{figure}

\section{$\IhJ$-divergence}

The $\I^*$-divergent sequences were studied in \cite{DASGHOSAL}. They were introduced as a generalization of statistically divergent sequences from \cite{MACSALSUBSEQ}. In this section we define $\IhJ$-divergent functions and we show that we can obtain similar results as the results given in \cite{DASGHOSAL} for $\I^*$-divergent sequences

Let $(X,d)$ be a metric space. We introduce the following definitions.

\begin{DEF}
Let $\I$ be an ideal on a set $S$. A function $\Zobr fSX$ is said to be \emph{$\I$-divergent} if there is an $x\in X$ such that for any $r>0$, arbitrarily large,
$$\Invobr f{\ol B(x,r)} = \{k\in S; d(f(k),x)\le r\}\in\I$$
or, equivalently,
$$\Invobr f{X\sm \ol B(x,r)}= \{k\in S; d(f(k),x)>r\} \in \FI.$$
\end{DEF}

For $\I=\Fin$ and $S=\NN$ we get \cite[Definition 9]{DASGHOSAL} and for $S=\NN$ and any admissible ideal $\I$ on $\NN$ we get \cite[Definition 10]{DASGHOSAL}. For $\I=\Fin$ we simply say that $f$ is divergent.

Also note that if $\I$ is a proper ideal and $f$ is $\I$-divergent, then $f$ cannot be $\I$-convergent to any point of $X$. Actually if $A\subseteq S$, $A\notin\I$, then $f|_A$ also cannot $\I$-converge to any point of $X$.

\begin{DEF}
A function $\Zobr fSX$ is said to be \emph{$\I^*$-divergent} if there is an $M\in\FI$ such that $\Zobr{f|_M=g}MX$ is divergent (i.e., $\Fin$-divergent).
\end{DEF}

For $S=\NN$, this definition coincides with \cite[Definition 11]{DASGHOSAL}. We now extend this definition in the following manner.
\begin{DEF}
Let $\J$ and $\I$ be ideals on a set $S$. A function $\Zobr fSX$ is said to be \emph{$\IhJ$-divergent} if there is an $M\in\FI$ such that $\Zobr{f|_M=g}MX$ is $\J|_M$-divergent, i.e., there exists an $x\in X$ such that for any $r>0$, $\Invobr g{\ol B(x,r)}\in\J|_M$ (and so belongs to $\J$) or equivalently $\Invobr g{X\sm\ol B(x,r)}\in\FF(\J|_M)$, i.e., it is a set of the form $B\cap M$, where $B\in\FJ$.
\end{DEF}

If $\J=\Fin$, then the notion of $\IhJ$-divergence coincides with $\I^*$-divergence \cite{DASGHOSAL}.

\begin{LM}
If $\I$ and $\J$ are ideals on a set $S$ and $\Zobr fSX$ is a function such that $f$ is $\J$-divergent then $f$ is also $\IhJ$-divergent.
\end{LM}

The proof follows directly from the definitions.

For a given metric space $(X,d)$ we define a topological space $X^*$ on the set $X^*=X\cup\{\infty\}$, where $\infty\notin X$, in the following way:
We define a topology on $X\cup\{\infty\}$ by choosing the sets of the form $\{\infty\}\cup X\sm \ol B(x_0,r)$ for some fixed $x_0\in X$ to be a local base at the point $\infty$ and taking open neighborhoods of $x\in X$ from the original topology as a local base at $x$.

It is relatively easy to see that the resulting topology does not depend on the choice of $x_0$
and  that $\infty$ is an isolated point in $X^*$ if and only if the metric $d$ is bounded on $X$.
It is also useful to notice that $X\cup\{\infty\}$ is first countable.

From the definitions we immediately get the following result:
\begin{LM}
A function $\Zobr fSX$ is $\I$-divergent if and only it is $\I$-convergent to $\infty$ when considered as a function from $S$ to the space $X^*=X\cup\{\infty\}$.

A function $\Zobr fSX$ is $\IhJ$-divergent if and only it is $\IhJ$-convergent to $\infty$ when considered as a function from $S$ to the space $X^*=X\cup\{\infty\}$.
\end{LM}

Using the above lemma we can obtain some results on $\IhJ$-divergence directly from the corresponding results about $\IhJ$-convergence.

We know that if $\J\subseteq\I$, then $\IhJ$-convergence implies $\I$-convergence, see \cite[Proposition 3.7]{MACSLEIK}. For our purposes the following form of converse of this result will be useful.
\begin{LM}\label{LMDIVER1}
Let $\I$, $\J$ be ideals on a set $S$. Let $X$ be a Hausdorff topological space and $x\in X$ be a non-isolated point. Suppose that there exists at least one function $\Zobr gSX$ such that $x\notin\Obr gS$ and $g$ is $\J$-convergent to $x$.

If every function $\Zobr fSX$ such that $x\notin\Obr fS$ which is $\IhJ$-convergent to $x$ is also $\I$-convergent to $x$, then $\J\subseteq\I$.
\end{LM}

\begin{proof}
Suppose $\J\nsubseteq\I$, i.e., there exists a set $A\in\J\sm\I$. Let $y\ne x$. Let us define
$$f(t)=
  \begin{cases}
 g(t) & \text{if }t\notin A, \\
    y & \text{if }t\in A.
  \end{cases}
$$
Clearly, $x\notin\Obr fS$.

For any neighborhood $U$ of $x$ we have $\Invobr fU \supseteq \Invobr gU\sm A$.
Since $g$ is $\J$-convergent to $x$, this implies $\Invobr fU\in\FJ$, and thus $f$ is also $\J$-convergent to $x$. Consequently, $f$ is also $\IhJ$-convergent to $x$.

If we take a neighborhood $U$ of $x$ such that $y\notin U$, then we have $\Invobr fU\subseteq S\sm A\notin\FI$. This shows that $f$ is not $\I$-convergent to $x$.
\end{proof}

\begin{THM}\label{THMDIV1}
\noindent
\begin{enumerate}
\enu
  \item Let $\I$ and $\J$ be two ideals on a set $S$ and $\J\subseteq\I$. If $\Zobr fSX$ is $\IhJ$-divergent then it is also $\I$-divergent.
  \item Let $(X,d)$ be a metric space such that $d$ is not bounded. Let $\I$ and $\J$ be ideals on a set $S$  such that there is at least one $\J$-divergent function from $S$ to $X$. If $\IhJ$-divergence implies $\I$-divergence, then $\J\subseteq\I$.
\end{enumerate}
\end{THM}

\begin{proof}
The first part the proof of \cite[Proposition 3.7]{MACSLEIK}. The second part follows from Lemma \ref{LMDIVER1}.
\end{proof}

As a special case of the above theorem we get \cite[Theorem 2]{DASGHOSAL}.

In general $\I$-divergence does not imply $\IhJ$-divergence even if $\J\subseteq\I$ as can be seen from \cite[Example 2]{DASGHOSAL}, where a function $\Zobr f{\NN}X$ (where $X=\RR$ with the usual metric) is constructed which is $\I$-divergent but not $\IhJ$-divergent, where $\I$ is a particular admissible ideal on $\NN$ (see \cite{DASGHOSAL} for details) and $\J=\Fin$.

In the next two results we show that the condition $\APIJ$ is both necessary and sufficient for the implication
$$\text{$\I$-divergence $\Ra$ $\IhJ$-divergence}$$
under some general conditions.

From \cite[Theorem 3.11]{MACSLEIK}  we get
\begin{THM}\label{THMDIV2}
Let $\I$ and $\J$ be two ideals on a set $S$ and $\I$ satisfies  the additive property with respect to $\J$, i.e. $\APIJ$ holds. Let $X$ be a metric space. Then for any function $\Zobr fSX$, $\I$-divergence implies $\IhJ$-divergence.
\end{THM}

To get the partial converse of the above result we again prove some auxiliary result about $\IhJ$-convergence:
\begin{LM}\label{LMDIVER2}
Let $X$ be a first countable Hausdorff topological space, let $\I$, $\J$ be ideals on $S$.
Let $x\in X$ be a point such that there exists a function $\Zobr gSX$ such that $x\notin\Obr gS$ and $g$ is
$\I$-convergent to $x$.

Suppose that every function $\Zobr fSX$ such that $x\notin\Obr fS$, which is $\I$-convergent to $x$, is also $\IhJ$-convergent to $x$. Then the ideals $\I$, $\J$ fulfill the condition $\APIJ$.
\end{LM}

\begin{proof}
Since $X$ is first countable and $x$ is not an isolated point, there exists a sequence $x_n$ of points from $X\sm\{x\}$, which is convergent to $x$.

Let $\{A_n, n\in\NN\}$, be a system of mutually disjoint sets from $\I$. Let us define a function $\Zobr fSX$ as
$$f(s)=
  \begin{cases}
    x_n; & s\in A_n, \\
    g(s); & s\notin \bigcup_{n\in\NN}A_n,.
  \end{cases}
$$
For every neighborhood $U$ of $x$ there exists an $n\in\NN$ such $\Invobr f{X\sm U}\subseteq \Invobr g{X\sm U} \cup \bigcup\limits_{k=1}^n A_k$, which implies that $\Invobr f{X\sm U}\in\I$. This shows that $f$ is $\I$-convergent to $x$.

By the assumptions of the lemma this implies that $f$ is also $\IhJ$-convergent to $x$. I.e., there is a set $M\in\FI$ such that $f|_M$ is $\J|_M$-convergent to $x$. This means that for every neighborhood $U$ of $x$ we have
$$\Invobr f{X\sm U}\cap M = A\cap M$$
for some $A\in\J$. In particular, this implies $\Invobr f{X\sm U}\cap M\in\J$.

Let us define $B_i=A_i\sm M$. We have $\bigcup\limits_{i\in\NN} B_i \subseteq S\sm M \in\I$.

At the same time, for the set $B_i\triangle A_i=A_i\cap M$ we have
$$A_i\cap M \subseteq \Invobr f{X\sm U} \cap M$$
for any neighborhood $U$ of $x$ such that $x_i\notin U$.
Consequently $B_i\triangle A_i \in \J$.
\end{proof}

Note that Lemma \ref{LMDIVER2} is, to some extent, similar to \cite[Theorem 3.12, Remark 3.13]{MACSLEIK}. The difference is that here we are working only with maps which do not attain value $x$.

From the above lemma we get:
\begin{THM}\label{THMDIV3}
Let $(X,d)$ be a metric space such that $d$ is not bounded. Let $\I$ and $\J$ be two ideals on $S$ such that $\J\subseteq\I$ and there exists at least one $\J$-divergent function from $S$ to $X$. If for every function $\Zobr gSX$, $\I$-divergence of $g$ implies $\IhJ$-divergence of $g$ then $\APIJ$ holds.
\end{THM}

Note that Theorem 3 and Theorem 4 of \cite{DASGHOSAL} are just special cases of Theorem \ref{THMDIV2} and Theorem \ref{THMDIV3} taking $S=\NN$, $\J=\Fin$ and $\I$ any admissible ideal on $\NN$.

\noindent{\textbf{Acknowledgement.}}
The first author is thankful to INSA (India) and SAS (Slovakia) for arranging a visit under bilateral exchange program when this work was done.
Part of this research was carried out while M. Sleziak was a member of the research team 1/0330/13 supported in part by the grant agency VEGA (Slovakia).

\bibliographystyle{amsplain}
\bibliography{martin}

%
%
%
%
%
%
%
%
%
%
%
%
%
%

\end{document}